\documentclass[11pt]{article}

\usepackage{amsmath}
\usepackage{amssymb}
\usepackage{amsthm}
\usepackage{titlesec}
\usepackage{graphicx}
\usepackage{caption}
\usepackage{subcaption}

\usepackage[T1]{fontenc}
\usepackage[utf8]{inputenc}

\usepackage{comment}

\usepackage{indentfirst}
\usepackage[top=1.25in,left=1.25in,right=1.25in]{geometry}
\newlength{\numone}
\newlength{\widone}
\newlength{\numtwo}
\newlength{\widtwo}
\newcounter{countp}

\newtheorem{thm}{Theorem}%[section]

\newtheorem{prop}[thm]{Proposition}
\theoremstyle{definition}
\newtheorem{defin}[thm]{Definition}
\newtheorem{alg}[thm]{Algorithm}
\newtheorem{rem}[thm]{Remark}

\numberwithin{equation}{section}

\author{\Large{Riccardo W. Maffucci\footnote{\textsc{EPFL, MA SB Batiment 8, Lausanne, Switzerland. }\texttt{riccardo.maffucci@epfl.ch}.}}}
\title{\Large{\uppercase{\bf Self-dual polyhedra of given degree sequence}}}

\date{}

\def\calP{\mathcal{P}}

\def\calS{\mathcal{S}}

\begin{document}
\titleformat{\section}
  {\Large\scshape}{\thesection}{1em}{}
\titleformat{\subsection}
  {\large\scshape}{\thesubsection}{1em}{}
\maketitle

%\pagenumbering{roman}
%\addcontentsline{toc}{section}{Table of contents}
%\tableofcontents

\begin{abstract}
Given vertex valencies admissible for a self-dual polyhedral graph, we describe an algorithm to explicitly construct such a polyhedron. Inputting in the algorithm permutations of the degree sequence can give rise to non-isomorphic graphs.

As an application, we find as a function of $n\geq 3$ the minimal number of vertices for a self-dual polyhedron with at least one vertex of degree $i$ for each $3\leq i\leq n$, and construct such polyhedra. Moreover, we find a construction for non-self-dual polyhedral graphs of minimal order with at least one vertex of degree $i$ and at least one $i$-gonal face for each $3\leq i\leq n$. 
	
%Let $n\geq 3$ and $r_n$ be a $3$-polytopal graph such that for every $3\leq i\leq n$, $r_n$ has at least one vertex of degree $i$. We find the minimal vertex count for $r_n$. We then describe an algorithm to construct the graphs $r_n$. A dual statement may be formulated for faces of $3$-polytopes. The ideas behind the algorithm generalise readily to solve related problems.

%Moreover, given a $3$-polytope $t_l$ comprising a vertex of degree $i$ for all $3\leq i\leq l$, $l$ fixed, we define an algorithm to output for $n>l$ a $3$-polytope $t_n$ comprising a vertex of degree $i$, for all $3\leq i\leq n$, and such that the initial $t_l$ is a subgraph of $t_n$. The vertex count of $t_n$ is asymptotically optimal, in the sense that it matches the aforementioned minimal vertex count up to order of magnitude, as $n$ gets large. In fact, we only lose a small quantity on the coefficient of the second highest term, and this quantity may be taken as small as we please, with the tradeoff of first constructing an accordingly large auxiliary graph.
\end{abstract}

%\tableofcontents

\section{Introduction}
\subsection{Results}
\label{sec:res}
This paper is about topological properties of polyhedra, namely the number of edges incident to a given vertex (degrees or valencies of vertices), and the number of faces adjacent to a given face (`degrees' or valencies of faces).

The $1$-skeleton of a polyhedron is a planar, $3$-connected graph -- the Rademacher-Steinitz Theorem. These graphs are embeddable in a sphere in a unique way (an observation due to Whitney). We will call them polyhedral graphs, or polyhedra for short. The dual graph of a polyhedron is a polyhedron. Vertex and face valencies swap in the dual. We call a polyhedron self-dual if it is isomorphic to its dual. A self-dual polyhedron on $p$ vertices has $p$ faces and $2p-2$ edges (straightforward consequence of Euler's formula).

In \cite{mafpo2} we considered the problem of minimising the number of vertices of a polyhedron containing at least one vertex of valency $i$, for each $3\leq i\leq n$. We established, among other results, that the minimal order  (i.e. number of vertices) for such graphs is
\[\left\lceil\frac{n^2-11n+62}{4}\right\rceil, \qquad n\geq 14.\]
The dual problem, that has therefore the same answer, is about imposing instead that there is at least one $i$-gonal face, for each $3\leq i\leq n$. In this paper, we assume both conditions.

\begin{defin}
	\label{def:1}
We say that a polyhedron $G$ has the \textit{property $\calS_n$} if it comprises at least one vertex of degree $i$ for every $3\leq i\leq n$, \textit{and} at least one $i$-gonal face for every $3\leq i\leq n$.
\end{defin}

Our first consideration is that if we ask instead for the minimal number of \textit{faces}, and assume we have such a graph $G$, then its dual $G^{*}$ also satisfies $\calS_n$, and has minimal vertices. Thereby, the answer to both questions must be the same. Moreover, it is natural to also seek self-dual solutions.

\begin{thm}
	\label{thm:1}
Let $n\geq 3$ and $G$ be a polyhedral graph satisfying $\calS_n$. Then the minimal number of vertices of $G$ is
\begin{equation}
\label{eq:vf}
\frac{n^2-5n+14}{2}, \qquad\forall n\geq 3.
\end{equation}
Moreover, Algorithm \ref{alg:1} constructs for each $n\geq 6$ a non-self-dual polyhedron $H_n$ of order \eqref{eq:vf} satisfying $\calS_n$, whereas Algorithm \ref{alg:2} constructs for each $n\geq 3$ a self-dual polyhedron $G_n$ of order \eqref{eq:vf} satisfying $\calS_n$. The speed of the said algorithms is quadratic in $n$, i.e., linear in the graph order.
\end{thm}

Theorem \ref{thm:1} will be proven in section \ref{sec:follow}. The construction of the self-dual solutions is a special case of the following more general result, to be proven in section \ref{sec:gen}.

\begin{thm}
	\label{thm:2}
	Let $k\geq 0$ and
	\begin{equation}
	\label{eq:seq}
	%a_1^{m_1},a_2^{m_2},\dots,a_{k-1}^{m_{k-1}},3^m
	t_1,t_2,\dots,t_{k},3^m
	\end{equation}
	be given \footnote{The notation $3^m$ indicates that the number $3$ is repeated $m$ times.}, where the integers $t_i$ are not necessarily distinct, each $t_i\geq 4$, and
	\begin{equation}
	\label{eq:m}
	m=4+\sum_{i=1}^{k}(t_i-4).
	\end{equation}
	%Then there exists a self-dual polyhedral graph $G$ of degree sequence \eqref{eq:seq}.
	Then Algorithm \ref{alg:2} constructs a self-dual polyhedral graph of degree sequence \eqref{eq:seq}. Inputting in Algorithm \ref{alg:2} a permutation of the $t_i$ produces, in general, non-isomorphic graphs. The speed of the algorithm is linear in the graph order.
\end{thm}
\begin{rem}
For fixed $t_1,\dots,t_k$, we need equality \eqref{eq:m} to hold in order for \eqref{eq:seq} to be the degree sequence of a self-dual polyhedron. Indeed, we have
\[\sum_{i=1}^{k}t_i+3m=2q=2(2p-2)=4(k+m)-4\]
by the handshaking lemma and self-duality. Algorithm \ref{alg:2} thereby constructs a self-dual polyhedral graph for any given admissible degree sequence.
\end{rem}

\subsection{Discussion and related work}
Theorem \ref{thm:1} solves a natural modification of the questions investigated in \cite{mafpo2}, as mentioned in section \ref{sec:res}. The method is to establish a lower bound on the minimal order of graphs satisfying the property $\calS_n$, and then to explicitly construct, for each $n$, solutions of such order via an algorithm. Here the expression for the minimal order \eqref{eq:vf} is cleaner, and the constructions more straightforward than in \cite{mafpo2}. Theorem \ref{thm:1} will be proven in section \ref{sec:follow}. %and Appendix \ref{appa}, independent of Theorem \ref{thm:2}.
%We will also see that the self-dual construction of Theorem \ref{thm:1} is an application of Theorem \ref{thm:2}.
The self-dual construction of Theorem \ref{thm:1} is an application of Theorem \ref{thm:2}.
	
Theorem \ref{thm:2} is about constructing self-dual polyhedra for any admissible degree sequence. The notions of duality and self-duality have been investigated since antiquity, with the Platonic solids. However, it was only relatively recently that the cornerstone achievement of generating all self-dual polyhedra was carried out \cite{arcric}. This was done by constructing all their \textit{radial graphs}, to be defined in section \ref{sec:gen}. Indeed, there is a one-to-one correspondence between self-dual polyhedra and their radial graphs. 

Their radial graphs are certain $3$-connected quadrangulations of the sphere (i.e. polyhedra where all faces are cycles of length $4$), namely, those with no {\em separating $4$-cycles} (i.e. all $4$-cycles are faces). Self-duals and these quadrangulations are thereby intimately related (there is a caveat, a $3$-connected quadrangulation of this type is not necessarily the radial of a {\em self-dual} polyhedron). Now, the generation of all quadrangulations of the sphere is another cornerstone result in graph theory \cite{br2005,bata89}. Equipped with this knowledge, we will prove Theorem \ref{thm:2} (section \ref{sec:pf}).

\paragraph{Notation.}
We will usually denote vertex and edge sets of a graph $G$ by $V(G)$ and $E(G)$, and their cardinality by $p=|V(G)|$ (order) and $q=|E(G)|$ (size). We will work with simple graphs (no loops or multiple edges).
\\
For $p\geq 4$, we call $W_p$ the $p-1$-gonal pyramid (or wheel graph), of $p$ \textit{vertices}.
\\
Let $\calP$ be an operation on a graph $G$, that modifies a given subgraph $H$ of $G$. The notation $\calP(G)$ is not well-defined as $G$ may contain no subgraph isomorphic to $H$, or may be ambiguous when the choice of $H$ is not unique. Given the graphs $G,G'$, we will write $\calP[G]\cong G'$ when there exists a subgraph $H$ of $G$ such that the graph obtained from $G$ on applying $\calP$ to $H$ is isomorphic to $G'$.

\paragraph{Acknowledgements.}
The author was supported by Swiss National Science Foundation project 200021\_184927.

%\subsection{Data availability statement}
%All data generated or analysed during this study are included in this article.

\section{Proof of Theorem \ref{thm:1}}
\label{sec:follow}

\paragraph{Lower bound.}
Let the graph $G$ satisfy property $\calS_n$. In particular, $G$ has at least one vertex of valency $i$ for every $3\leq i\leq n$. As shown in \cite[proof of Lemma 7]{mafpo2}, we then have a lower bound on the edges $q=|E(G)|$,
\begin{equation*}
%\label{eqn:lb}
2q\geq\frac{(n-2)(n-3)}{2}+3p.
\end{equation*}
%The graphs $r_n$ constructed in \cite[Algorithm 8]{mafpo2} satisfy $\calP_n$ of \cite[Definition 1]{mafpo2}, i.e., they have at least one vertex of valency $i$ for $3\leq i\leq n$, and are of minimal size among polyhedra satisfying this condition. The $r_n$ have all triangular faces.
%For the present problem, we also want (at least) one $i$-gonal face for $4\leq i\leq n$. Therefore,
On the other hand, imposing that $G$ has at least one $i$-gonal face for all $4\leq i\leq n$ yields
\begin{equation*}
2q\leq 6p-12-2\sum_{i=4}^{n}(i-3)=6p-12-(n-3)(n-2).
\end{equation*}
Combining the two inequalities yields the lower bound in Theorem \ref{thm:1}
\begin{equation*}
p\geq\frac{n^2-5n+14}{2}, \qquad\forall n\geq 3.
\end{equation*}

\paragraph{Construction.} We now turn to actually constructing $3$-polytopes of such order. Consulting \cite[Table I]{fede75}, we find that entries $1$ (tetrahedron), $2$ (square pyramid) and $34$ (Figure \ref{fig:fo1}) are the unique polyhedra of minimal order satisfying $\calS_3$, $\calS_4$, and $\calS_5$ respectively. These are all self-dual. Next, we construct for each $n\geq 6$ a non-self-dual polyhedron of minimal order satisfying $\calS_n$.

\begin{figure}[h!]
	\centering
	\begin{subfigure}{0.28\textwidth}
		\centering
		\includegraphics[width=2.5cm,clip=false]{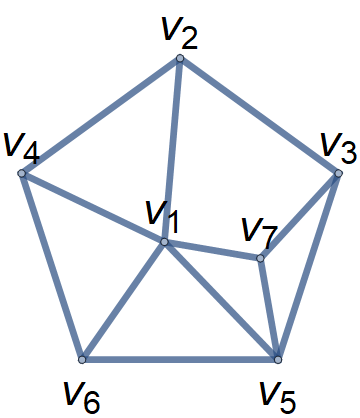}
		\caption{The only polyhedron $H_5=G_5$ of minimal order satisfying $\calS_5$.}
		\label{fig:fo1}
	\end{subfigure}
	\hspace{1.0cm}
	\begin{subfigure}{0.62\textwidth}
		\centering
		\includegraphics[width=8cm,clip=false]{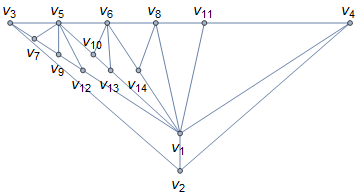}
		\caption{The polyhedron $H_7$ constructed in Algorithm \ref{alg:1}.}
		\label{fig:fo2}
	\end{subfigure}
	\caption{}
	\label{fig:follow}
\end{figure}

\begin{alg}\
	\label{alg:1}
	
	\textbf{Input.} A natural number $N\geq 6$.
	
	\textbf{Output.} For each $6\leq n\leq N$, a non-self-dual polyhedron $H_n$ of minimal order satisfying $\calS_n$.
	
	\textbf{Description.} We start by considering the graph $H_5$ in Figure \ref{fig:fo1} with its attached vertex labelling, by setting the integer $n:=6$, and the set of $n-3$ triples
	\[
	S:=\{(v_1, v_4, v_6), (v_5, v_1, v_7), (v_6, v_1, v_5)\}.
	\]
	At each step, given $H_{n-1}$, we perform the operation depicted in Figure \ref{fig:split}, `edge splitting', to each vertex triple of $S$ in turn, taking for $u_1,u_2,u_3$ the entries of the triple in order. We label successively $v_8,v_9,\dots$ the newly inserted vertices via the edge splitting. This yields the graph $H_n$. The graph $H_7$ is illustrated in Figure \ref{fig:fo2}. At the same time, we modify $S$ in the following way. Upon applying edge splitting to $(a,b,c)$, say, we replace it by the new triple $(a,b,v)$, where $v$ is the new vertex introduced by the splitting. Lastly, calling $a'$ the first vertex of the last triple in $S$, we insert the further triple $(v_{|V(H_{n-1})|+1},v_1,a')$, and increase $n$ by $1$. The algorithm stops as soon as $n=N+1$.
\end{alg}

\begin{figure}
[h!]
\centering
\begin{subfigure}{0.4\textwidth}
	\centering
	\includegraphics[width=2.5cm,clip=false]{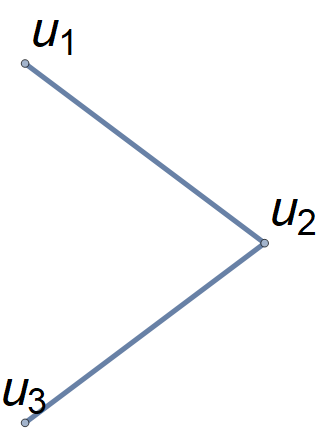}
	%\caption{The only polyhedron $H_5=G_5$ of minimal order satisfying $\calS_5$.}
	\label{fig:split1}
\end{subfigure}
\hspace{0.5cm}
$\longrightarrow$
\hspace{0.5cm}
\begin{subfigure}{0.4\textwidth}
	\centering
	\includegraphics[width=2.5cm,clip=false]{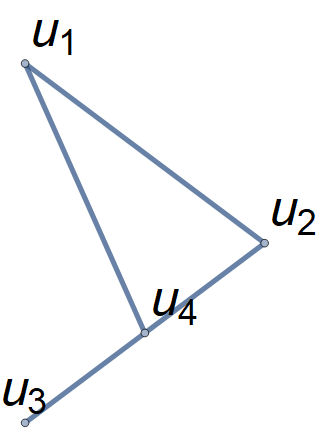}
	%\caption{The polyhedron $H_7$ constructed in Algorithm \ref{alg:1}.}
	\label{fig:split2}
\end{subfigure}
\caption{Edge splitting operation on the vertices $(u_1,u_2,u_3)$, consecutive on the boundary of a face.}
\label{fig:split}
\end{figure}

\begin{rem}
Edge splitting has the effect of raising by one the valencies of the vertex $u_1$ and of the face containing $u_2,u_3$ but not $u_1$. It also introduces the new vertex $u_4$ of degree $3$, and the new triangular face $u_1,u_2,u_4$.
\end{rem}

It is straightforward to check by induction, with base case $n=6$, that the $H_n$ of Algorithm \ref{alg:1} indeed satisfy the sought properties of Theorem \ref{thm:1}. First, edge splitting is well-defined, as it is always performed on a triple of vertices forming a triangular face. Indeed, once we replace $(u_1,u_2,u_3)$ of Figure \ref{fig:split} with $(u_1,u_2,u_4)$ as in the algorithm, the latter triple forms a face. As for the last triple inserted at each step, note that it is simply
\[(v_{|V(H_{n-1})|+1},v_1,v_{|V(H_{n-2})|+1}).\] 
The vertices $u_1=v_{|V(H_{4})|+1}=v_6$, $u_2=v_1$, $u_3=v_{|V(H_{3})|+1}=v_5$ form a triangle in $H_5$. Therefore, after edge splitting, $u_1$, $u_2$, and $u_4=v_{|V(H_{5})|+1}=v_8$ are the vertices of a triangle in $H_6$, and so forth in this fashion.

Second, for the graph order \eqref{eq:vf}, each step adds $n-3$ vertices, and we have by induction
\[\frac{(n-1)^2-5(n-1)+14}{2}+n-3=\frac{n^2-5n+14}{2}.\]

Third, to obtain $H_n$ from $H_{n-1}$, we perform $n-3$ edge splittings. These transform a vertex of degree $i$ into one of degree $i+1$, for $3\leq i\leq n-1$ respectively. Moreover, at the same time an $i$-gon gets replaced by an $i+1$-gon: indeed, in Figure \ref{fig:split} the face different containing $u_2,u_3$ but not $u_1$ loses the edge $u_2u_3$ and acquires $u_2u_4, u_4u_3$. We conclude that $H_n$ satisfies $\calS_n$.

Fourth, we show that $H_n$ is not self-dual for any $n\geq 6$. On one hand, $\deg_{H_n}(v_1)=n$, $\deg_{H_n}(v_5)=n-1$, and $v_1v_5\not\in E(H_n)$. On the other hand, in $H_n$ the $n$-gon and the $n-1$-gon share the edge $v_2v_3$.

Lastly we note that Algorithm \ref{alg:1} may be implemented in linear time in the graph order (quadratic in $n$).

%\begin{rem}
%The duals $H_n^*$ form another sequence of polyhedra with the same properties.
%\end{rem}

\begin{rem}
There are several other constructions, similar to Algorithm \ref{alg:1}, yielding polyhedra of minimal order satisfying $\calS_n$, e.g. the duals $H_n^*$. The idea is to apply $n-3$ edge splittings at each step, where each simultaneously increases by $1$ the valency of a vertex and of a face.
\end{rem}

\paragraph{The self-dual case, assuming Theorem \ref{thm:2}.} For the last part of Theorem \ref{thm:1} we require the further condition of self-duality. However, constructions with edge splitting in general do not preserve the self-duality of $H_5$ in the new graphs obtained from it. % Algorithm \ref{alg:3} to follow yields polyhedra of minimal order satisfying $\calS_n$, and also preserves self-duality.
In the next section we will present Algorithm \ref{alg:2}, that produces a self-dual polyhedron for any given admissible degree sequence, as stated in Theorem \ref{thm:2}. The self-dual polyhedra of Theorem \ref{thm:1} may be constructed independently of the arguments of section \ref{sec:gen}, although possibly in a less intuitive fashion. Here we complete the proof of Theorem \ref{thm:1} assuming Theorem \ref{thm:2}. To obtain $G_n$ we simply input the tuple $(4,5,\dots,n)$, i.e. the sequence
\[n,n-1,\dots,4,3^{(n^2-7n+20)/2},\]
into Algorithm \ref{alg:2}.

%Algorithm \ref{alg:3} actually generalises readily to Algorithm \ref{alg:2}, that produces a self-dual polyhedron for any given admissible degree sequence. In section \ref{sec:gen}, we will introduce quadrangulations and the radial graph to present Algorithm \ref{alg:2} and prove Theorem \ref{thm:2}. The self-dual polyhedra of Theorem \ref{thm:1} may be constructed independently of the arguments of section \ref{sec:gen}, although possibly in a less intuitive fashion. We relegate Algorithm \ref{alg:3} to Appendix \ref{appa} for the interested reader.

%Algorithms \ref{alg:1} and \ref{alg:3} both apply edge splittings, starting from the same graphs. The former does not preserve self-duality, while the latter does. A more precise understanding of this comes from analysing radial graphs and quadrangulations of the sphere in the next section.

\section{Generating self-dual polyhedra}
\label{sec:gen}
\subsection{Radial graphs and quadrangulations}
\label{sec:th}
The \textit{radial}, or \textit{vertex-face} graph $R_G$ of a plane graph $G$ is obtained by taking $V(R_G)$ to be the set of vertices and regions of $G$. We have an edge between two vertices $u,v$ of $R_G$ whenever $u$ is a vertex of $G$, and $v$ a region of $G$, such that $u$ lies on the boundary of $v$ in $G$ \cite[section 2.8]{mohtho}.

If the plane graph $G$ is $2$-connected, the newly constructed $R_G$ is a \textit{quadrangulation} of the sphere, i.e. each region is delimited by a $4$-cycle \cite[section 2.8]{mohtho}. If $G$ is a polyhedron then so is $R_G$ \cite[Lemma 2.1]{arcric}. Moreover, $G$ is a polyhedron if and only if $R_G$ has no separating $4$-cycles (i.e. $4$-cycles that are not faces, so that removing the cycle disconnects the graph) \cite[Lemma 2.8.2]{mohtho}.% In \cite{arcric} it was proven that all self-dual polyhedra may be obtained with one of six constructions. This was done by constructing all corresponding radial graphs.

The radial graph of the tetrahedron is the cube, and more generally the radial graph of the pyramid (or wheel) $W_p$, $p\geq 4$ is the so-called `pseudo double wheel' $PDW_{2p}$ (of $2p$ vertices), i.e. the dual graph of the $p-1$-gonal antiprism. As established in \cite[Theorem 3]{br2005}, and initially stated in \cite{bata89}, all polyhedral quadrangulations of the sphere are obtained from the cube by applying three transformations $\calP_1,\calP_2,\calP_3$, sketched in \cite[Figure 3]{br2005} and \cite[Figure 3]{bata89}.
We introduce the notation $\bf{C}(\mathfrak{G},\mathfrak{P})$ for the set of all graphs that may be obtained from an initial set of graphs $\mathfrak{G}$ by applying the set of transformations $\mathfrak{P}$. Under this notation, the previous statement may be rephrased as, \[{\bf{C}}(\{PDW_8\}, \{\calP_1,\calP_2,\calP_3\})
\text{ is the set of $3$-connected quadrangulations of the sphere.}\]

Moreover,
\[{\bf{C}}(\{PDW_{2p} : p\geq 4\}, \{\calP_1\})\]
is the set of all polyhedral quadrangulations without separating $4$-cycles \cite[Theorem 4]{br2005}. It follows that
\[{\bf{C}}(\{PDW_{2p} : p\geq 4\}, \{\calP_1\})\quad
\textit{ is the set of radial graphs of polyhedra.}\]

We note that the transformation $\calP_2$ replaces a subgraph of $G$ that is isomorphic to $PDW_8-v$ with a copy of $PDW_{10}-v$. In particular, $\calP_2[PDW_{2p}]\cong PDW_{2p+2}$.
Therefore, we have
\[{\bf{C}}(\{PDW_{2p} : p\geq 4\}, \{\calP_1\})\subseteq{\bf{C}}(\{PDW_{8}\}, \{\calP_1,\calP_2\}).\]

%Further, the radial graph of a polyhedron cannot have any separating $4$-cycles $a,b,c,d$ say, otherwise $\{a,c\}$ would form a cut-set of (a polyhedron isomorphic to) either $G$ or $G^{*}$, that is impossible.

%radial graphs of polyhedra may be obtained from the cube via $\calP_1,\calP_2$

For $G$ a \textit{self-dual} polyhedron, we have in particular $|V(R_G)|=2|V(G)|$ and $|E(R_G)|=2|E(G)|=2|V(R_G^*)|$. Furthermore, we can recover $G$ from $R_G$ by noting that the latter is always bipartite, and taking for $G$ all of the vertices in either part of $R_G$, together with edges for $G$ between pairs of vertices belonging to the same face in $R_G$. The above considerations have the following consequence.

\begin{prop}
The radial graph of any polyhedron $G$ may be obtained from the cube via the transformations $\calP_1,\calP_2$ of \cite[Figure 3]{br2005}. %In particular, this gives a way to generate, among other graphs, all self-dual polyhedra.
Moreover, the number of applications of $\calP_1$ to generate self-duals is even.
%let $n\geq 3$, and $R_G$ be the radial graph of a self-dual polyhedron $G$ of minimal order satisfying $\calS_n$. Then $R_G$ is obtained from the cube by applying $\calP_1,\calP_2$, where the number of applications of $\calP_1$ is even.
\end{prop}
\begin{proof}
By the arguments of the present section, it suffices to prove that when $G$ is self-dual, the number of applications of $\calP_1$ on the cube to obtain $R_G$ is indeed even. From \cite[Figure 3]{br2005}, we observe that $\calP_1$ has the effect of adding an edge to $G$, and a vertex and an edge to $G^{*}$. As opposed to this, $\calP_2$ adds one vertex and one edge to both $G,G^{*}$. We have thus obtained our parity argument.
\end{proof}

In the next section we prove Theorem \ref{thm:2}, putting it in the context of the above literature.

\subsection{The proof of Theorem \ref{thm:2}}
\label{sec:pf}

As it turns out, for any $n\geq 3$, generating a self-dual polyhedron $G$ of minimal order satisfying $\calS_n$ may be done by applying only a transformation $\calP$ (to be defined below, and similar to $\calP_2$ of \cite{br2005,bata89})
to the cube in order to construct $R_G$, and then passing to $G$. This generalises readily to Theorem \ref{thm:2}, as we will now prove.%proven in the next subsection.

We begin by defining a function $f$, that maps a tuple $T=(t_1,t_2,\dots,t_k)$, $k\geq 0$, of integers $\geq 4$ to the degree sequence \eqref{eq:seq}
\begin{equation*}
f(T)= t_1,t_2,\dots,t_{k},3^m,
\end{equation*}
where $m$ is given by \eqref{eq:m}.

\begin{alg}\
	\label{alg:2}
	
	\textbf{Input.} A $k$-tuple of integers $T=(t_1,t_2,\dots,t_k)$, with $t_i\geq 4$ for each $i$.
	
	\textbf{Output.} A self-dual polyhedron $G(T)$ of degree sequence $f(T)$.
	
	\textbf{Description.}
We will construct the radial graph $R_{G(T)}$, and then pass to $G(T)$ as explained in section \ref{sec:th}. We begin by setting $R_{G(T)}$ to be the cube $PDW_8$, radial graph of the tetrahedron. We also consider a subgraph $H$ of $R_{G(T)}$ with the vertex labelling of Figure \ref{fig:p2a}. We define the transformation $\calP$ that modifies a subgraph $H$ of a graph $G$ as shown in Figure \ref{fig:p2}. %\[\mathbf{v}=(v_1,v_2,v_3,v_4,v_5):=(a,b,c,A,C)\]
%using the labelling of Figure .
\\
We stop when $T$ is empty. Each step entails $t_i-3$ successive applications of $\calP$ to $R_{G(T)}$. Before each subsequent application, we apply to $H$ a graph isomorphism $\varphi$ such that
\begin{equation}
\label{eq:phi}
\begin{array}{ccc}
\varphi(a)=a,
&\varphi(b)=c,
&\varphi(c)=d,
%&\varphi(d)=b,
\\
\varphi(A)=A,
&\varphi(B)=C,
&\varphi(C)=D.%,
%&\varphi(D)=B.
\end{array}
\end{equation}
as labelled in Figure \ref{fig:p2}. Following all the $t_i-3$ operations, we instead apply to $H$ the graph isomorphism $\psi$ satisfying
\begin{equation}
\label{eq:psi}
\begin{array}{cccc}
\psi(a)=c,
&\psi(b)=a,
&\psi(c)=d,
%&\psi(d)=b,
\\
\psi(A)=C,
&\psi(B)=A,
&\psi(C)=D.
%&\psi(D)=B,
\end{array}
\end{equation}
then we delete $t_i$ from $T$, and proceed to the next step.
\end{alg}

\begin{figure}
	[h!]
	\centering
	\begin{subfigure}{0.4\textwidth}
		\centering
		\includegraphics[width=2.5cm,clip=false]{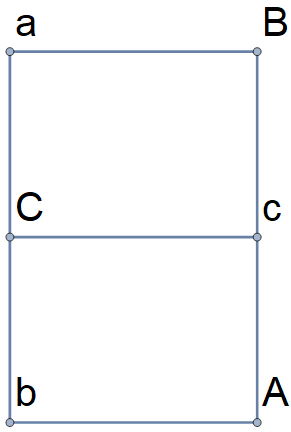}
		\caption{Subgraph $H$ of $R_G$.}
		\label{fig:p2a}
	\end{subfigure}
	\hspace{0.5cm}
	$\longrightarrow$
	\hspace{0.5cm}
	\begin{subfigure}{0.4\textwidth}
		\centering
		\includegraphics[width=2.5cm,clip=false]{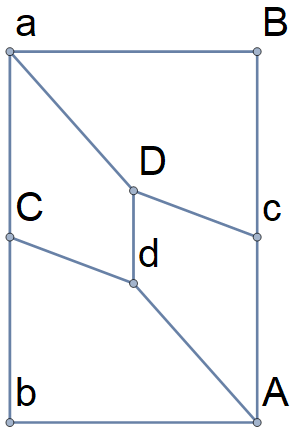}
		\caption{Resulting subgraph.}
		\label{fig:p2d}
	\end{subfigure}
	\caption{The transformation $\calP$. Vertices of $G$ have lower-case labels, those of $G^*$ upper-case.}
	\label{fig:p2}
\end{figure}

%\begin{rem}
	%There are in general several polyhedra of given degree sequence \eqref{eq:seq}. Algorithm \ref{alg:2} does not construct them all. However, in many cases permutations of the $t_i$'s give rise to non-isomorphic solutions.
%\end{rem}

\begin{rem}
	%\label{rem:cor}
There are in general several polyhedra for a  given degree sequence \eqref{eq:seq}. Algorithm \ref{alg:2} does not construct them all. On the other hand, in many cases permutations of the $t_i$'s give rise to non-isomorphic solutions, as may be observed via direct computation.% \cite{comp01}.
\end{rem}

\begin{rem}
It follows from Theorem \ref{thm:2} that the set \[\mathbf{C}(\{PDW_8\},\{\calP\})\]
contains the radial graph of at least one self-dual polyhedron for any given degree sequence. As for how many radial graphs of self-dual polyhedra of given size belong to $\mathbf{C}(\{PDW_8\},\{\calP\})$, we have computed the values of Table \ref{tab:1} for small sizes (data available on request).
\begin{table}[h!]
	\centering
	%\begin{equation}
	$\begin{array}{|c||c|c|c|c|c|c|c|c|c|}
	\hline
	\text{Size }q&6&8&10&12&14&16&18&20&22\\\hline
	\text{Radials of self-duals in } \mathbf{C}(\{PDW_8\},\{\calP\})&1
	&1&2&5&15&40&140&417&1496\\\hline
	\text{Total self-duals}&1
	&1&2&6&16&50&165&554&1908\\\hline
	\end{array}$
	%\end{equation}
	\caption{For given graph size $q$, the number of radial graphs $R_G$ of self-dual polyhedra $G$ with size $q$ that belong to $\mathbf{C}(\{PDW_8\},\{\calP\})$, compared to the total. For values in the last row, refer e.g. to \cite{dillen}.}
	\label{tab:1}
\end{table}
\end{rem}

\begin{rem}
\label{rem:pp2}
The transformation $\calP$ is similar to $\calP_2$ of \cite{br2005,bata89}. More precisely, $\calP_2$ is applicable if and only if, $\calP$ may be applied and moreover either $a,b$ or $A,B$ belong to the same face in $R_G$ (referring to the labelling of Figure \ref{fig:p2}).
%if $\calP_2$ is applicable to a graph $G$, then so is $\calP$.
\end{rem}

\begin{rem}
\label{rem:tr}
Applying $\calP$ to $R_G$ has the same effect on $G$ and $G^*$ as applying the edge splitting of Figure \ref{fig:split} to them, where $u_1=a$, $u_2=b$, $u_3=c$, and $u_4=d$ (and analogously for vertices $A,B,C,D$ of $G^*$).
\end{rem}

Let us now complete the proof of Theorem \ref{thm:2}. We start by justifying applicability of the transformation $\calP$. The initial cube clearly has a subgraph isomorphic to $H$ in Figure \ref{fig:p2a}. Furthermore, the graph in Figure \ref{fig:p2d} also has a subgraph isomorphic to $H$, where the isomorphism is $\varphi$ \eqref{eq:phi}. The same statement remains true for $\psi$ \eqref{eq:psi}.

Starting with the cube $R_{G((3,3,3,3))}$, each operation $\calP$ clearly yields another $3$-connected quadrangulation of the sphere. We now check that self-duality of $G$ is preserved by the algorithm. Each operation $\calP$ on $R_{G(T)}$ transforms $G(T)$ and $G^{*}(T)$ in the same way (Remark \ref{rem:tr}). As the initial $G((3,3,3,3))$ (tetrahedron) is self-dual, then so will all the successive $G(T)$'s be. Further, the following considerations for lower-case labels $a,b,c,d$ apply verbatim to the upper-case ones by duality.

We now analyse how each step affects the degrees of the vertices in $G$. First, the degree of a vertex in $G$ is the number of faces that the corresponding vertex lies on in $R_G$, i.e., $\deg_G(v)=\deg_{R(G)}(v)$ for each $v$ by the definition of radial graph. Now, each application of $\calP$ adds $1$ to the degree of $a$ (and $A$ of $G^{*}$), introduces the new vertex $d$ (and $D$ of $G^{*}$), of degree $3$, and leaves other valencies unchanged.  When we update $H$ via $\varphi$ \eqref{eq:phi}, $a$ is mapped to itself. Therefore, step $i=1,\dots,k$ has the effect of increasing by $t_i-3$ the degree of $a$ (and $A$).

Second, we claim that the algorithm step $i$ increases by $t_i-4$ the number of vertices of valency $3$ in $G$. By the considerations above, the first application of $\calP$ increases one valency of $G$ from $3$ to $4$, and adds a new vertex of degree $3$. Hence the first application of each step leaves the number of vertices of valency $3$ in $G(T)$ unchanged. Each subsequent application of $\calP$ increases their total by $1$. Now step $i$ entails $t_i-3$ operations of type $\calP$, hence the number of vertices of degree $3$ increases by $(t_i-3)-1$ as claimed.

Third, we claim that, at the beginning of each algorithm step, in $G(T)$ with its attached labelling one has
\begin{equation*}
%\label{eq:d3}
\deg(a)=\deg(A)=3.
\end{equation*}
We show this claim by induction. In the initial cube all vertices are of valency $3$. When we apply $\psi$ \eqref{eq:psi} to $H$, $a$ is mapped to $c$, and $\deg(c)=3$ since $\calP$ does not modify its degree.

Putting everything together, after $k$ algorithm steps the degree sequence of $G(T)$ will be
\[t_1,t_2,\dots,t_k,3^{4+\sum_{i=1}^{k}(t_k-4)}\]
i.e., at the end of the algorithm the resulting sequence will be \eqref{eq:seq} as desired.

As for algorithm speed, the total number of operations to obtain $G(T)$ is proportional to the sum of the $t_i$'s, i.e. to the graph size $q$, that is to say, to its order $p$ since $q=2p-2$. The proof of Theorem \ref{thm:2} is complete.

\paragraph{Future work.}
Our investigation generates a portion of the self-dual polyhedra (recall Table \ref{tab:1}), starting from the tetrahedron, by applying $\calP$ to its radial graph (Figure \ref{fig:p2}). This portion includes at least one such graph for every admissible degree sequence. It would be of interest to further analyse the set $\mathbf{C}(\{PDW_8\},\{\calP\})$ and its properties.

%It would be of interest to find an alternative way to \cite{arcric} for generating all self-duals.
%\\
%Another (related) question is to determine, if possible, a topological property shared by the elements of $\mathbf{C}(\{PDW_8\},\{\calP\})$ that are radials of self-duals.

\clearpage

\clearpage
\bibliographystyle{plain}
\bibliography{bibgra}

\end{document}